\documentclass[12pt]{amsart}

\calclayout
\usepackage{stmaryrd}
\usepackage{amsfonts}
\usepackage{mathtools}
\usepackage{amssymb}
\usepackage{wrapfig,lipsum,booktabs}
\usepackage{placeins}
\usepackage[hang]{subfigure}
\usepackage{amsmath} 
\usepackage{times}
\usepackage{floatrow}
\usepackage[utf8]{inputenc}
\usepackage[english]{babel}
\usepackage{amsthm}
\usepackage{subfigure} 
\usepackage{blindtext}
\usepackage{tkz-fct}
\usepackage[margin=1in]{geometry} 
\usepackage{tikz}
\usepackage{tikz-cd}
\usetikzlibrary{mindmap,trees}
\usepackage[final]{pdfpages}
\newtheorem{thm}{Theorem}
\newtheorem*{remark}{Remark}
\newtheorem*{lem}{Lemma}
\theoremstyle{definition}
\newtheorem*{definition}{Definition}
\newtheorem*{corollary}{Corollary}
\newtheorem{case}{Case}

\usepackage{times} 
\usepackage{floatrow}
\newfloatcommand{capbtabbox}{table}[][\FBwidth]
\usepackage{blindtext}
\usepackage{setspace}
\author{Yazan Alamoudi}
\address{Department of Mathematics\\
University of Florida\\
Gainesville\\
FL 32611\\
United States
}
\email{yazanalamoudi@ufl.edu}

\dedicatory{Dedicated to George Andrews and Bruce Berndt.}

\begin{document}
\title{ A bijective proof of Andrews' refinement of the Alladi-Schur theorem}

\begin{abstract}
This paper gives a bijective proof of Andrews' refinement of the Alladi-Schur theorem. Moreover, it demonstrates that the bijective framework introduced here can be used to reproduce and provide a bijective account of Andrews' recursive relations for the Alladi-Schur polynomials.
\end{abstract}

\maketitle
\section{Introduction}
In private communications, see \cite{Andrews2019}, K. Alladi informed G. Andrews of a new variant of Schur's partition theorem that deserved closer study. That variant, which Andrews dubbed the Alladi-Schur theorem, is the following. The number of partitions of $n$ into odd parts, each occurring at most twice, is equivalent to the number of Schur partitions of $n$. Recall that Schur partitions are partitions where the gap between consecutive parts is $\geq3$ and no consecutive multiples of $3$ occur. Subsequently, Andrews was able to refine the result to what follows.\\

Denote by $m(\pi)$ the number of parts \textbf{plus} the number of even parts of the partition $\pi$. Let $\mathcal{D}(m,n)$ denote the set of Schur partitions of $n$ with $m(\pi)=m$. Let $\mathcal{C}(m,n)$ denote the set of partitions of $n$ into $m$ odd parts, each occurring at most twice. For the remainder of this paper, we will call members of $\mathcal{C}(m,n)$ Alladi partitions. Then, Andrews was able to show that
\begin{equation}
  |\mathcal{D}(m,n)|=|\mathcal{C}(m,n)|  
\end{equation}
  by means of generating functions \cite{Andrews2019}.
More specifically, in \cite{Andrews2019}, Andrews established equation (1) by showing that
\begin{equation}
  \sum_{n,m\geq0} D(m,n)x^mq^n=\prod_{n=1}^\infty(1+xq^{2n-1}+x^2q^{4n-2}),
\end{equation}

where $|\mathcal{D}(m,n)|=D(m,n)$.\\

Here, I will establish equation (1) by giving a bijection from the set $\mathcal{D}(m,n)$ to the set $\mathcal{C}(m,n)$.\\

This paper is structured as follows. In Section 2, I will introduce and prove some of the key ingredients of the main result. Moreover, Section 2 will give the reader a general idea of the approach taken to construct the bijection. Sections 3 and 4 are the most technical, with Section 4 featuring the main result in Theorem 5. In Section 5, I will apply the framework developed in the earlier sections to obtain new, simpler proofs of Andrews' recursive relations, see \cite{Andrews2019} and \cite{10.1007/978-3-319-68376-8_3}, of the Alladi-Schur polynomials. Furthermore, we will see that these proofs give a complete bijective account of these recursive identities. Lastly, there is an appendix that is neither essential nor used in the body of the paper. However, it provides alternative perspectives that may be helpful for the reader.
\bigskip

\textbf{Note:}
Based on the recommendations of the editorial office, the author has reduced the use of the active voice (including in the abstract) in favor of the royal ``we" and the passive voice, as it is more consistent with standard practice. The reader should keep this in mind whenever tonal shifts occur, as the author is attempting to balance editorial recommendations with his natural voice.
\section{The bijective components and first results}

To understand many ideas in this paper, it helps to make a few simple observations. Let $p(x)$ denote the infinite product in equation $(2)$. Then, it is easy to see that 
\begin{equation}
p(x)=(1+xq+x^2q^2)p(xq^2). 
\end{equation}

For the Alladi partitions, the recursive equation $(3)$ is easy to understand. If $\pi\in\mathcal{C}(m,n)$, first, remove any occurrence of the part $1$, then subtract $2$ from every part in the resulting partition.

\bigskip

Now, let $\mathcal{C^*}(m,n)\subseteq\mathcal{C}(m,n)$ be the set of Alladi partitions where every part is $>1$. Denote the map removing any occurrence of the part $1$ by $G'$ and denote the map subtracting $2$ from every part by $\psi'$ and let $\phi'=\psi'G'$ (i.e. $\phi'(\pi)=\psi'(G'(\pi))$ for all $\pi\in\mathcal{C}(m,n)$). We have the following three bijections. 
\begingroup
\setlength{\jot}{15pt}
\begin{gather*}
    G':\mathcal{C}(m,n)\to \mathcal{C^*}(m,n)\cup\mathcal{C^*}(m-1,n-1)\cup\mathcal{C^*}(m-2,n-2),\\
    \psi':\mathcal{C^*}(m,n)\to\mathcal{C}(m,n-2m),\\
    \phi':\mathcal{C}(m,n)\to \mathcal{C}(m,n-2m)\cup\mathcal{C}(m-1,n-2m+1)\cup\mathcal{C}(m-2,n-2m+2).
\end{gather*}
\endgroup

In other words, $\phi'$ corresponds to equation $(3)$ for the Alladi partitions. Call $G'$ the Alladi grouping map and $\psi'$ the Alladi reduction\footnote{In earlier drafts and presentations of this paper, the map $\psi'$ was referred to as the Alladi projection map, as that is what I formerly named it. Likewise, I formerly named the analogous map $\psi$, constructed in Section 4, the Schur projection map and the map $\rho$, also constructed in Section 4, the projection map on the minimal segments. Although the term ``projection", as previously used in the name of these maps, originates from earlier versions of this paper, I settled on ``reduction" as a more appropriate term. This is because it is typical to use the term projection for idempotent maps.} map. A big part of what I will do is construct corresponding maps for the Schur partitions. We will begin this with the following bijection.\footnote{For the remainder of this paper, the sum of the parts of a partition $\pi$ is denoted by $\sigma(\pi)$.}

\begin{lem}{(Schur grouping map)}
Let $\mathcal{D^*}(m,n)$ denote the subset of partitions in $\mathcal{D}(m,n)$ with the added condition that every odd part is $\geq 3$ and every even part is $\geq 6$.
There is a bijection $$G:\mathcal{D}(m,n)\to \mathcal{D^*}(m,n)\cup\mathcal{D^*}(m-1,n-1)\cup\mathcal{D^*}(m-2,n-2).$$

\end{lem}
\begin{proof}
First, define the Schur grouping map $G$ on $\mathcal{D}(m,n)$ as follows. For $\pi \in \mathcal{D}(m,n)$, remove any part less than $3$; in addition, replace any occurrence of the part $4$ with $3$. This way\\ $G(\pi)\in \mathcal{D^*}(m,n)\cup\mathcal{D^*}(m-1,n-1)\cup\mathcal{D^*}(m-2,n-2)$. It's clear that $G$ is invertible. Indeed, take  $\pi'\in \mathcal{D^*}(m,n)\cup\mathcal{D^*}(m-1,n-1)\cup\mathcal{D^*}(m-2,n-2)$. If $\sigma(\pi')=n$, then we already have $G(\pi')=\pi'$. If $\sigma(\pi')=n-1$ and $3$ is a part, we replace the $3$ with $4$ to obtain the preimage of $\pi'$. Otherwise, we just insert $1$ as a part. Lastly, if $\sigma(\pi')=n-2$ and $3$ is a part, we replace the $3$ with $4$ and insert $1$ as a part to obtain the preimage of $\pi'$. Otherwise, we just insert $2$ as a part. This completes the proof. 
\end{proof}

\begin{remark}
We note a subtlety with both grouping maps.\footnote{See also the remark after the corollary to Theorem 5.} It almost seems like if you apply the procedure twice, you get the same thing. However, it is important to note that the definition of the inverse of the grouping map, and thus the grouping map as a bijection, depends on the domain of $G$. So it is more accurate to write $G_{m,n}$ and $G'_{m,n}$ instead of $G$ and $G'$, where, as usual, $m=m(\pi)$ and $n=\sigma(\pi)$. In general, it is not true to say that $G_{m,n}(G_{m,n}(\pi))=G_{m,n}(\pi)$ for every $\pi\in\mathcal{D}(m,n)$ or that $G'_{m,n}(G'_{m,n}(\pi))=G'_{m,n}(\pi)$ for every $\pi\in\mathcal{C}(m,n)$ which would seem to produce contradictions. This is because, for fixed $m$ and $n$, $G(\pi)$ may not be in the domain of $G$ as the range and the domain are different. This is also true for $G'$. In other words, the expressions  $G_{m,n}(G_{m,n}(\pi))$ and $G'_{m,n}(G'_{m,n}(\pi))$ may not even be meaningful. For example, take the Alladi partition $\pi=1+3$. Then  $G'_{2,4}(1+3)=3$, but now the partition $3$ has one part. That is to say, the expression $G'_{2,4}(G'_{2,4}(1+3))$ has no meaning, and you can't conclude that $G'_{2,4}$ is the identity, arriving at a contradiction. However, we can simply write $G$ and $G'$ for convenience, as $m$ and $n$ will be clear from context, and no ambiguity shall arise. In particular, when we write $G(\pi)$ we always mean $G_{m(\pi),\sigma(\pi)}(\pi)$ and similarly for $G'$. Nonetheless, clarifying remarks are made when appropriate, as is the case for the corollary to Theorem 5. \end{remark}

My guiding intuition is to view the maps $G$ and $G'$ as being dual to each other. The following theorem supports this perspective.

\begin{thm}
 $$|\mathcal{D^*}(m,n)|=|\mathcal{C^*}(m,n)|,$$  and consequently, $$|\mathcal{D^*}(m,n)|=|\mathcal{D}(m,n-2m)|.$$ 
\end{thm}

\begin{proof}

 We proceed by strong induction on $m$. This is trivial for $m\leq 1$. Suppose this is true for every $m<M$.  By the lemma and equation $(1)$  $$|\mathcal{D^*}(M,n)|+|\mathcal{D^*}(M-1,n-1)|+|\mathcal{D^*}(M-2,n-2)|=|\mathcal{C^*}(M,n)|+|\mathcal{C^*}(M-1,n-1)|+|\mathcal{C^*}(M-2,n-2)|.$$ 
  Using the induction hypothesis and canceling
$$|\mathcal{D^*}(M,n)|=|\mathcal{C^*}(M,n)|.$$ 

Consequently, from the Alladi reduction map and equation $(1)$ we have $$|\mathcal{D^*}(M,n)|=|\mathcal{C^*}(M,n)|=|\mathcal{C}(M,n-2M)|=|\mathcal{D}(M,n-2M)|.$$
\end{proof}
From here, it can be seen that there \textbf{exists some} bijection sending $\mathcal{D^*}(m,n) \xrightarrow{\sim} \mathcal{D}(m,n-2m)$. I will construct a more refined map, which we call the Schur reduction map, that has that effect. From there, the bijection for the main theorem will be defined recursively.\\

For the remainder of this paper, the parts of a partition are separated by a $+$ sign whenever the context is clear. If there is any ambiguity, or just to be safe, the parts are separated by a comma or a $\oplus$ symbol.\\

We make one last note. Let $[\mathcal{C}(m,n)]_g$ denote the set of Alladi partitions with $g$ ones. On the other side, let $[\mathcal{D}(m,n)]_1$ denote the set of Schur partitions that contain a $1$ or $4$ but not both. Additionally, let $[\mathcal{D}(m,n)]_2$ denote the set of Schur partitions that begin with a $2$ or $1\oplus 4$. Lastly, set $[\mathcal{D}(m,n)]_0=\mathcal{D^*}(m,n)$. From the grouping maps and the previous theorem, it is clear that 
\begin{corollary}
    $$|[\mathcal{D}(m,n)]_g|=|[\mathcal{C}(m,n)]_g|.$$
\end{corollary}

A big part of the remainder of this paper will essentially be dedicated to the construction of the Schur reduction map. Before we proceed with this task, we must discuss the important notion of a minimal segment\footnote{K. Alladi prefers and has suggested the term ``fundamental segment" instead. To my mind, both ``minimal segment" and ``fundamental segment" seem good. However, I have settled on ``minimal segment" for this work.} and provide other preliminary definitions. The basic definitions resemble those of Kurşungöz \cite{KURSUNGOZ2021112563}, but they quickly diverge. For example, the notions of singletons and pairs appear in \cite{KURSUNGOZ2021112563} but are not identical to mine as I will use a different pairing convention. Moreover, even singletons are split into free and minimal segment singletons, a distinction not present in \cite{KURSUNGOZ2021112563}. Most importantly, the concept of a minimal segment is new.

\section{Minimal segments and other preliminary definitions}

We will begin with what we call the \underline{generic upper factorization} of a Schur partition $\pi\in\mathcal{D}^*(m,n)$. Here, an integer $x\in \pi$ is part of an \underline{upper pair} if $x+3$ or $x-3$ is also a part with two exceptions. The first is that $x$ is odd, and it is the largest part of an odd-length maximal arithmetic progression with a common difference of $3$. The second exception is that $x$ is even, and it is the smallest part of an odd-length maximal arithmetic progression with a common difference of $3$. In other words, if we have a maximal odd-length arithmetic progression, with a common difference of $3$,  of $1 \pmod{3}$ or $2 \pmod{3}$ parts, then we pair smallest first if the endpoints are odd numbers and we pair largest first if the endpoints are even. We remind the reader that the definition of Schur partitions precludes consecutive multiples of $3$, so they will never be part of pairs. An integer that is not part of an upper pair is called an \underline{upper singleton} in the context of the generic upper factorization. \\

For our purposes, we will need a more refined factorization. To do this, I introduce the crucial concept of a minimal segment and provide two equivalent definitions. One of these definitions will be presented in the appendix.
\begin{definition}
A  \underline{primitive upper minimal segment} is a Schur partition taking on one of the following two forms.\\   

1- The smallest part is an odd singleton $> 3$, which begins an arithmetic progression (length $=1$ is allowed) of odd singletons with a common difference of $4$. This is followed by an arithmetic progression (length $=1$ is allowed) of even singletons, with a common difference of $4$, whose smallest part is $5$ more than the largest odd singleton. Specifically, this form looks like
$$o, o + 4, . . . , o + 4x, o + 4x + 5, o + 4x + 9, . . . , o + 4(x + y) + 5.$$ Here, $x,y\geq 0$. Note that both $x=0$ and $y=0$ are allowed.\\

2-The smallest part is an odd singleton $> 3$, which begins an arithmetic progression (length $=1$ is allowed) of odd singletons with a common difference of $4$. This is followed by a nonempty even-length arithmetic progression of $1 \pmod 3$ or $2\pmod 3$ paired parts, with a common difference of $3$, whose smallest part is $4$ more than the largest odd singleton. Then, an arithmetic progression (length $=1$ is allowed) of even singletons, with a common difference of $4$, whose smallest part is $4$ more than the largest part of a pair.  To illustrate, let $p^* = o + 4x + 1$, then, this form looks like
$$o, o + 4, \dots , o + 4x, [p^*+3,p^*+6] \dots , [p^*+3(p-1),p^*+3p], p^* + 3p+4, \dots, p^* +3p+4y.$$

Notice that in this case, we need $p^*=o + 4x +1 \not\equiv 0 \pmod{3}$ as well as $x\geq0$ and $y,p>0$ with $p$ even. Note also that all of $x=0$, $y=1$, and $p=2$ are allowed. \\
\newpage
Now, define the set of upper minimal segments $\mathcal{U}$ as the smallest set of Schur partitions that contains all primitive upper minimal segments and is closed under the following two operations.\\

Rule 1- (Suffix rule) If $\Sigma\in\mathcal{U}$ is an upper minimal segment with largest part $e$ an even integer such that $e\not\equiv 0 \pmod{3}$, then the union 
of $\Sigma$ with $$[e+3,e+6],. . .,[e+3(p-1),e+3p],e+3p+4,\dots,e+3p+4y$$ is also an upper minimal segment. Here, $p,y>0$, and $p$ is even.\\

Rule 2- (Prefix rule) If $\Sigma\in\mathcal{U}$ is an upper minimal segment with smallest part\\ $o=o' + 4k + 3(p+1)+1\not\equiv 0 \pmod{3}$ with $o'>3$ an odd integer, then the union 
of $\Sigma$ with $$o', o'+4, \dots,o' +4k, [o' +4k +3+1, o' +4k +6+1],\dots, [o' +4k +3(p-1)+1,o' +4k +3p+1]$$ is also an upper minimal segment. Here, $k\geq 0$, $p>0$, and $p$ is even.  \\

Now, an \underline{upper minimal segment} is a subpartition $\Sigma\subseteq\pi$ of a Schur partition $\pi$ such that $\Sigma\in\mathcal{U}$ \textbf{and} the endpoints of $\Sigma$ are singletons in the generic upper factorization of $\pi$. For example,  $28+31+36$ is treated as a pair $[28,31]$ and a singleton $36$ even though $31+36\in\mathcal{U}$. This is because, in the generic upper factorization of $\pi$, one of the endpoints of $31+36$  (namely $31$) is already paired up.

\end{definition}

Notice that an upper minimal segment always begins with an odd singleton and ends with an even singleton. Furthermore, every pair has its smallest part odd. In general, the gap is always $3$ or $4$ except for a possible single exception of a gap of $5$  where an even singleton comes immediately after an odd singleton. In addition, the first and last gaps are always $\geq 4$.\\

We now put it all together. Upper singletons that are part of an upper minimal segment will be called \underline{UMS-singletons}. Upper singletons that are not part of an upper minimal segment will be called \underline{free upper singletons}. A similar distinction can be made with pairs. However, we do not emphasize this as the Schur reduction map, introduced in the next section, will have the same effect on pairs, whether they are free or part of a minimal segment. The \underline{refined upper factorization} will categorize parts as free upper singletons, parts in an upper minimal segment, and parts in an upper pair. We emphasize that in this categorization/factorization, we will look at \textbf{longest}\footnote{see ``Illustration of factorizations" later in this section.} upper minimal segments in the sense that they are not properly contained in other upper minimal segments. This categorization will partition each Schur partition $\pi\in\mathcal{D}^*(m,n)$ into intervals. We will call each of these intervals a refined factor or just a factor if the context is clear.\\

\newpage

\underline{Illustration of prefix and suffix rules:}\\

The primitive upper minimal segment $\Sigma= 35+40$ has  $40\not\equiv0 \pmod{3}$ so we can add the suffix $[43,46]+50$. The resulting minimal segment still has the largest part a non-multiple of $3$, so we can add another suffix  $[53+56]+[59+62]+66+70$. This looks like 
$$35+40\to35+40+[43,46]+50$$ $$\to35+40+[43,46]+50+[53+56]+[59+62]+66+70.$$
 Adding a prefix looks like $$35+40+[43,46]+50+[53+56]+[59+62]+66+70$$
$$\to25+[29+32]+35+40+[43,46]+50+[53+56]+[59+62]+66+70.$$

\bigskip

In an analogous way, we have the \underline{generic lower factorization} of a Schur partition. Here, an integer $x>1$ is part of a \underline{lower pair} if $x+3$ or $x-3$ is also a part greater than $1$, with two exceptions. The first is that $x$ is even, and it is the largest part of an odd-length maximal arithmetic progression of parts greater than $1$ with a common difference of $3$. The second exception is that $x$ is odd, and it is the smallest part of an odd-length maximal arithmetic progression of parts greater than $1$ with a common difference of $3$. In other words, for the generic lower factorization, we adopt the opposite pairing convention. That is, for an odd length maximal arithmetic progression (not including $1$ if it appears as a part) of $1 \pmod{3}$ or $2 \pmod{3}$ parts, we pair smallest first if the endpoints are even numbers, and we pair largest first if the endpoints are odd. An integer that is not part of a pair is called a \underline{lower singleton}. Note that if the part $1$ is present, it is always a lower singleton by definition.\\ 

Now, let $h$ be the map that subtracts $3$ from every part. Define a  \underline{lower minimal segment} as a partition (or sub-partition) whose largest and smallest elements are lower singletons and is of the form $h(\Sigma)$ with $\Sigma$ an upper minimal segment. Lower singletons that are part of a lower minimal segment will be called \underline{LMS-singletons}. Lower singletons that are not part of a lower minimal segment will be called \underline{free lower singletons}. A similar distinction can be made with pairs. However, this will not be emphasized for the same reason as the one given before with upper pairs. The \underline{refined lower factorization} will categorize parts as free lower singletons, parts in a lower minimal segment, and parts in a lower pair.  Again, it is important to emphasize that in this categorization, we will look at \textbf{longest} lower minimal segments in the sense that they are not properly contained in other lower minimal segments.\\

Notice that the definition of lower minimal segments makes them inherit analogous properties and characterizations to those of upper minimal segments. In particular, they also have a prefix and suffix rule. Moreover, if $\Sigma$ is a primitive upper minimal segment, then we call $h(\Sigma)$ a primitive lower minimal segment.\\

\underline{Illustration of factorizations:}\\

The Schur partition $$\pi=12+17+21+25+28+31+34+37+42+60+70+73$$ 
has upper factorization $$\pi=12+17+21+[25+28]+[31+34]+37+42+60+[70+73]$$
refined upper factorization
$$\pi=12+\underline{17+21+[25+28]+[31+34]+37+42}+60+[70+73]$$
lower factorization
$$\pi=12+17+21+25+[28+31]+[34+37]+42+60+[70+73]$$
and refined lower factorization
$$\pi=\underline{12+17+21+25}+[28+31]+[34+37]+42+60+[70+73].$$

In the above, brackets are put around pairs and minimal segments are underlined. In general, generic upper factorizations and generic lower factorizations can differ only in bracketing odd-length maximal arithmetic progression of $1 \pmod 3$ or $2 \pmod{3}  $ parts. The difference between generic and refined is only in underlining the minimal segments. However, the lower and upper minimal segments corresponding to the same partitions could look completely different. Notice also that it is not correct to say that $\pi$ has the refined upper factorization $$12+17+\underline{21+[25+28]+[31+34]+37+42}+60+[70+73]$$ even though $21+[25+28]+[31+34]+37+42$ is an upper minimal segment. This is because it is not a longest upper minimal segment as it is properly contained in $17+21+[25+28]+[31+34]+37+42$, which is also an upper minimal segment. The same applies to refined lower factorizations. \\

 \smallskip
 
Now notice that, for an upper or lower minimal segment, if we replace every odd singleton with the letter $O$,
every even singleton with the letter $E$, every element in a pair with the letter $P$, and every addition with non-commutative multiplication, we obtain a codeword in these three letters. If we set $O=E=S$, we obtain a codeword in two letters. Call the former the codeword and the latter the $S$-codeword. For example,
the upper minimal segment $5+10+14$ has codeword $OE^2$ and $S$-codeword $S^3$. On the other hand, the upper minimal segment $7+[11+14]+18$ has codeword $OP^2E$ and $S$-codeword $SP^2S$. Observe that adding a suffix $[17+20]+24$ to $5+10+14$ corresponds to adding the suffix $P^2E$ to $OE^2$.\\

 In general, every upper (or lower) minimal segment is uniquely retrievable from the quadruple $(W,\omega,\epsilon,s)$ where $W$ is the $S$-codeword,  $\omega$ is the number of odd parts, $\epsilon$ the number of even parts, and $s$ is the starting (smallest) part. \\ 

\section{Bijections}
With the groundwork laid by the last section, the bijections can now be presented.

\begin{thm}(Reduction map on the minimal segments)
Let $\mathcal{U}(\omega,\epsilon,p_1,p_2,n)$ denote the set of upper minimal segments of total sum $n$ with $\omega$ odd parts, $\epsilon$ even parts, $p_1$ parts of a $1\pmod{3}$ upper pair, and $p_2$ parts of a $2\pmod{3}$ upper pair. Let $\mathcal{L}(\omega,\epsilon,p_1,p_2,n)$ denote the set of lower minimal segments of total sum $n$ with $\omega$ odd parts, $\epsilon$ even parts, $p_1$ parts of a $1\pmod{3}$ lower pair, and $p_2$ parts of a $2\pmod{3}$ lower pair. We have a bijection $$\rho:\mathcal{U}(\omega,\epsilon,p_1,p_2,n)\to\mathcal{L}(\omega,\epsilon,p_1,p_2,n-2(\omega+2\epsilon)).$$

In fact, an upper minimal segment with the quadruple $(W,\omega,\epsilon,s)$ is sent to a lower minimal segment with the quadruple $(W,\omega,\epsilon,s-3)$.
 
\end{thm}

\begin{proof}
Notice that the fact that $(W,\omega,\epsilon,s)\to(W,\omega,\epsilon,s-3)$ actually defines the map. Furthermore, the choice of notation makes it plainly clear that it's a bijection that preserves $\epsilon$, $\omega$, and $p_1+p_2=$ the number of occurrences of the letter $P$ in $W$. It remains to show that there exists a lower minimal segment with quadruple $(W,\omega,\epsilon,s-3)$. This is seen as follows. Recall that the map $h$ subtracting three from every part actually gives a bijection sending $(W,\epsilon,\omega,s)\to(W,\omega,\epsilon,s-3)$. In other words, it suffices to show that there exists an upper minimal segment with quadruple $(W,\epsilon,\omega,s)$. Now observe that if an upper minimal segment $\Sigma$ with quadruple $(W,\omega,\epsilon,s)$ has more than one even singleton, then subtracting $1$ from the smallest even singleton gives an upper minimal segment with quadruple $(W,\omega+1,\epsilon-1,s)$. A similar argument can be made with the largest odd singleton. This shows that there exists an upper minimal segment with quadruple $(W,\epsilon,\omega,s)$ and, hence, a lower minimal segment with quadruple $(W,\omega,\epsilon,s-3)$. The above argument is somewhat abstract, so allow me to give a more explicit demonstration below.\\

The map is simply to subtract $3$ from every part if $\omega=\epsilon$. Otherwise, let $\omega'$ and $\epsilon'$ denote the number of odd and even UMS-singletons, respectively. If $\omega>\epsilon$, then subtract $3$ from every part except the largest $\omega'-\epsilon'$ odd singleton parts, in which case you subtract $2$. Likewise, if $\omega<\epsilon$, then subtract $3$ from every part except the smallest $\epsilon'-\omega'$ even singleton parts, in which case you subtract $4$. Notice that the image is a lower minimal segment; every upper singleton is mapped to a lower singleton, and every upper pair is mapped to a lower pair (with the same congruence modulo 3) in an order-preserving way. In other words, an upper minimal segment with the quadruple $(W,\omega,\epsilon,s)$ is sent to a lower minimal segment with the quadruple $(W,\omega,\epsilon,s-3)$. Furthermore, the total weight is decreased by $2(\omega+2\epsilon)$. Indeed, suppose $\omega<\epsilon$ and let $d=\epsilon-\omega$ then we have subtracted $3\omega+3(\epsilon-d)+4d=6\omega+4d=2\omega+4\epsilon$. The case $\omega>\epsilon$ is similar. The map in the other direction, from the lower minimal segments, is analogous. The conclusion follows.
\end{proof}

\underline{Illustration of $\rho$ \#1:}
The map sends $$5+10+14\to2+6+11.$$
Notice that both have $S$-codewords $S^3$. The upper minimal segment on the LHS has codeword $OE^2$, and it is sent to the lower minimal segment on the RHS, which has codeword $E^2O$.\\

\underline{Illustration of $\rho$ \#2:}
The map sends $$5+9+14\to2+7+11.$$
Notice that both have $S$-codewords $S^3$. The upper minimal segment has codeword $O^2E$, and the lower minimal segment has codeword $EO^2$.\\

\underline{Illustration of $\rho$ \#3:}
In general, it can be seen that the upper minimal segment $$o,o+4,\dots,o+4x,o+4x+5,o+4x+9,\dots,o+4(x+y)+5$$ is sent to $$o-3,(o-3)+4,\dots,(o-3)+4y,(o-3)+4y+5,(o-3)+4y+9,\dots,(o-3)+4(x+y)+5.$$
Notice that both have $S$-codeword $S^{x+y+2}$. On the other hand, the upper minimal segment has codeword $O^{x+1}E^{y+1}$, whereas the lower has codeword $E^{y+1}O^{x+1}$.\\

\underline{Illustration of $\rho$ \#4:}

The upper minimal segment $$23+27+[31+34]+[37+40]+44+48+52+56$$ is sent to $$20+24+[28+31]+[34+37]+40+44+49+53.$$
 Notice that both have $S$-codeword $S^2P^4S^4$. On the other hand, the upper minimal segment has codeword $O^2P^4E^4$, whereas the lower  minimal segment has codeword $E^2P^4E^2O^2$.\\
 
Notice also that even though the upper minimal segment was primitive, it was sent to a lower minimal segment that is not primitive. \\

In general, if the upper minimal segment codeword has the first $i$ singleton factors $O$ and the rest $E$, then its image has the last $i$ singleton factors $O$ and the rest $E$.

\begin{thm}(The Schur reduction map)
Let $\mathcal{D^*}(\omega,\epsilon,p_1,p_2,n)$ denote the set of  Schur partitions of $n$ with the condition that every odd part is $\geq 3$ and every even part is $\geq 6$, having $\omega$ odd parts, $\epsilon$ even parts, $p_1$ parts of a $1\pmod{3}$ upper pair, and $p_2$ parts of a $2\pmod{3}$ upper pair. Let $\mathcal{D}(\omega,\epsilon,p_1,p_2,n)$ denote the set of  Schur partitions of $n$ with $\omega$ odd parts, $\epsilon$ even parts, $p_1$ parts of a $1\pmod{3}$ lower pair, and $p_2$ parts of a $2\pmod{3}$ lower pair.
There is a bijection $$\psi:\mathcal{D^*}(\omega,\epsilon,p_1,p_2,n)\to \mathcal{D}(\omega,\epsilon,p_1,p_2,n-2(\omega+2\epsilon))$$

and consequently, a bijection $$\psi:\mathcal{D^*}(m,n)\to \mathcal{D}(m,n-2m).$$ Here $m=m(\pi)$, and $\psi$ is used to denote both functions as the operation does not change but only the domain.
\end{thm}
\begin{proof}
I will adopt a method with similarities to Kurşungöz \cite{KURSUNGOZ2021112563}. Define the reduction map $\psi$ as follows. Let $\pi\in \mathcal{D^*}(m,n)$ (or $\pi\in \mathcal{D^*}(\omega,\epsilon,p_1,p_2,n)$). Then, by writing the refined upper factorization of $\pi$, the parts of $\pi$ are split into the following categories:

\begin{itemize}
    \item Parts in an upper $1 \pmod{3}$ pair
  \item Parts in an upper $2 \pmod {3}$ pair
  \item Parts in an upper minimal segment
  \item Free odd upper singletons 
  \item Free even upper singletons 
\end{itemize}
Paired parts are decreased by 3, free even singletons are decreased by 4, and free odd singletons are decreased by 2. In particular, for pairs, this looks like $$[3k+4,3k+7]\to [3k+1,3k+4] \;\;\;\;\text{and}\;\;\;\;[3k+5,3k+8] \to [3k+2,3k+5].$$ The map on the upper minimal segments is $\rho$. Although this follows from the refined upper factorization, it is important to emphasize that $\rho$ is applied on \underline{longest upper minimal segments} in the sense that they are not properly contained in any other upper minimal segments.

Note that this operation preserves $\epsilon$, $\omega$, $p_1$, and $p_2$. In addition, it decreases the total sum by $2m=2(\omega+2\epsilon)$. Furthermore, its image is a Schur partition.\\

  In the same manner,  the process will now be shown to be reversible. Indeed, as before, for $\pi\in\mathcal{D}(m,n-2m)$ (or $\pi\in\mathcal{D}(\omega,\epsilon,p_1,p_2,n-2(\omega+2\epsilon))$)  categorize parts of $\pi$ into the following categories:
\begin{itemize}
  \item Parts in a lower $1 \pmod 3$ pair \underline{with the smallest part bigger than 1}
  \item Parts in a lower $2 \pmod 3$ pair
  \item Parts in a lower minimal segment
  \item Free odd lower singletons 
  \item Free even lower singletons

\end{itemize}
 
Paired parts are increased by $3$, free even singletons are increased by $4$, and free odd singletons are increased by $2$. The map on the lower minimal segments is $\rho^{-1}$. As before, the previous map is applied on \underline{longest lower minimal segments} in the sense that they are not properly contained in other lower minimal segments. I call this map the Schur lift.\\

It's readily seen that these two operations are inverses of each other. 
\end{proof}

Observe that if $\{f_i\}$ is the refined upper factorization of $\pi$, then $\{\psi(f_i)\}$ is the refined lower factorization of $\psi(\pi)$. In addition, if $\pi_i$ is a free upper singleton, then $\psi(\pi_i)$ is a free lower singleton of the same parity. Moreover, if $\pi_i$ is a member of an upper pair, then $\psi(\pi_i)$ is a member of a lower pair with the same congruence class $\pmod{3}$. Finally, if $\pi_i$ is a UMS-singleton, then $\psi(\pi_i)$ is an LMS-singleton.\\

\underline{Illustration of $\psi$: }

The Schur partition \smallskip
$$\pi=15+20+35+40+43+46+50+55+60+73+77+82+85+90+94$$ \smallskip
has the following refined upper factorization. \smallskip
$$\pi=\underline{15+20}+\underline{35+40+[43+46]+50}+\underline{55+60}+73+77+[82+85]+90+94$$ \smallskip
This corresponds to the refined lower factorization of $\psi(\pi)$ shown below. \smallskip
$$\psi(\pi)=\underline{12+17}+\underline{32+36+[40+43]+47}+\underline{52+57}+71+75+[79+82]+86+90$$ \smallskip
What must be noted is the refined upper factorization of $\psi(\pi)$ is \smallskip
$$\psi(\pi)=12+17+32+36+[40+43]+\underline{47+52}+57+\underline{71+75+[79+82]+86+90}$$ \smallskip
which looks very different than both refined factorizations of $\pi$. However, this is what we must use to compute $\psi^2(\pi)$.  \smallskip
$$\psi^2(\pi)=8+15+28+32+[37+40]+\underline{44+49}+55+\underline{68+72+[76+79]+83+87}$$ \smallskip
In the above, the refined lower factorization is used to express $\psi^2(\pi)$. This illustrates how difficult it is to predict what repeated applications of $\psi$ will look like.\\

Now, the bijection\footnote{This bijection can be thought of as dual to $\phi'$.} on the Schur partitions that corresponds to equation  $(3)$ can finally be presented.

\begin{thm}
There is a bijection$$\phi:\mathcal{D}(m,n)\to \mathcal{D}(m,n-2m)\cup\mathcal{D}(m-1,n-2m+1)\cup\mathcal{D}(m-2,n-2m+2).$$
In particular,
$$D(m,n)=D(m,n-2m)+D(m-1,n-2m+1)+D(m-2,n-2m+2).$$

\end{thm}

\begin{proof}
It's clear that $\phi$, defined by $\phi(\pi)=\psi(G(\pi))$, is a bijection as it is a composition of bijections. Furthermore, the intersection of any two of the three sets on the right-hand side is empty since partitions from different sets have different sizes. So we have $$D(m,n)=D(m,n-2m)+D(m-1,n-2m+1)+D(m-2,n-2m+2)$$ as claimed.
\end{proof}
\bigskip
At last, the main result is now easy to prove. 
\bigskip

\begin{thm}
There is a bijection $$f:\mathcal{D}(m,n)\to \mathcal{C}(m,n).$$
\end{thm}

\begin{proof}
Let $r$ be the map that adds $2$ to every part. Let $g(\pi)=\sigma(\pi)-\sigma(G(\pi))$. The bijection is defined inductively by $f(\pi)=f(g(\pi))\oplus r(f(\phi(\pi)))$
with the initial conditions (base case) $f(0)=0$, $f(1)=1$, and $f(2)=1\oplus 1$. Here, we identify $0$ with the empty partition. Furthermore, we note that all of these maps are defined so as to send the empty partition to the empty partition. In particular, $r(0)=0$. First, observe that $\sigma(f(\pi))=\sigma(\pi)$ and $m(f(\pi))=m(\pi)$.
To prove the injectivity and surjectivity of $f$, we proceed by strong induction on $n$ (the integer being partitioned). \\
\newpage
Indeed, suppose this is injective for all partitions with $\sigma(\pi)\leq N$. If  
 $\sigma(\pi_1)=\sigma(\pi_2)=N+1$ and $f(\pi_1)=f(\pi_2)$ then 
  $$f(\pi_1)=f(\pi_2)\implies G'(f(\pi_1))=G'(f(\pi_2))\implies r(f(\phi(\pi_1)))=r(f(\phi(\pi_2)))$$ $$\implies f(\phi(\pi_1))=f(\phi(\pi_2))\xRightarrow{\text{By induction}}\phi(\pi_1)=\phi(\pi_2)\xRightarrow{\text{By Theorem 4}} \pi_1=\pi_2.$$

Here, $G'(f(\pi))=r(f(\phi(\pi)))$ was used. Now, in reality, injectivity proves bijectivity, as Andrews already established $|\mathcal{D}(m,n)|=| \mathcal{C}(m,n)|$ by generating functions. However, I will show surjectivity anyway.\\

We again proceed by strong induction on $n$ (the integer being partitioned). Suppose this is surjective for all partitions with $\sigma(\pi)\leq N$. If $\sigma(\pi)=N+1$,  then it can be seen that $\pi$ is the image of $\phi^{-1}(\alpha)=G^{-1}_{m(\pi),N+1}(\psi^{-1}(\alpha))$, where 
$\alpha\in f^{-1}(\{\phi'(\pi)\})$. The induction hypothesis guarantees that $f^{-1}(\{\phi'(\pi)\})$ is nonempty.  This completes the proof. 
\end{proof}
\begin{corollary}
We have 
$$f(\phi(\pi))=\phi'(f(\pi))\quad \text{   and } \quad f([\mathcal{D}(m,n)]_g)=[\mathcal{C}(m,n)]_g.$$

Moreover, the following diagrams commute.

\[
\begin{tikzcd}
\mathcal{D^*}(m,n)\arrow[d, "f"]  \arrow[r,"\psi"] &\mathcal{D}(m,n-2m)\arrow[d, "f"]    \\
 \mathcal{C^*}(m,n)\arrow[r, "\psi'"] &\mathcal{C}(m,n-2m)  
\end{tikzcd}
\begin{tikzcd}
\mathcal{D}(m,n)\arrow[d, "f"]  \arrow[r,"G"] &\mathcal{D^*}(m,n)\cup\mathcal{D^*}(m-1,n-1)\cup\mathcal{D^*}(m-2,n-2)\arrow[d, "f"]    \\
 \mathcal{C}(m,n)\arrow[r, "G'"] &\mathcal{C^*}(m,n)\cup\mathcal{C^*}(m-1,n-1)\cup\mathcal{C^*}(m-2,n-2)  
\end{tikzcd}
\]
\end{corollary}
\begin{proof}
 If $\pi^*\in\mathcal{D^*}(m,n)$ then $g(\pi^*)=0$ and $G(\pi^*)=\pi^*$. It follows $$f(\pi^*)=r(f(\phi(\pi^*)))=r(f(\psi(G(\pi^*))))=r(f(\psi(\pi^*)))$$ 
$$\implies\psi'(f(\pi^*))= f(\psi(\pi^*)).$$

Likewise, for $\pi\in \mathcal{D}(m,n)$, denote $\pi^*=G(\pi)$ then using the previous argument and noting $\phi(\pi)=\psi(\pi^*)$ we have
$$f(G(\pi))=f(\pi^*)=r(f(\psi(\pi^*)))=r(f(\phi(\pi)))=G'(f(\pi)).$$ The first two claims easily follow.
\end{proof}
\begin{remark}
Although it is not written down, it may look like the proof used $G(\pi)=G(\pi^*)$, which seems to contradict that $G$ is a bijection. In fact, there is no contradiction as the $G$ on the left-hand side may not be the same function as the $G$ on the right-hand side, as per the remark after the lemma concerning the Schur grouping map. So the previous equation, written more explicitly, is actually $G_{m,n}(\pi)=G_{m-g(\pi),n-g(\pi)}(\pi^*)$. It is important to note that the procedure of $G^{-1}$ \textbf{depends on the domain of G}. For example, even though for $\pi=1+4$ we have $G_{3,5}(\pi)=3=\pi^*=G_{1,3}(\pi^*)$, it is also true that $G_{3,5}^{-1}(3)=1+4$ and $G_{1,3}^{-1}(3)=3$.
\end{remark}

The above properties fit very nicely with the results at the beginning of this manuscript. In addition, they make computing the bijection a bit more convenient. This is illustrated below.

\[
\begin{tikzcd}
3+7+14\arrow[d, "f"]  \arrow[r,"\psi"] &1+5+10\arrow[d, "f"]  \arrow[r,"G"] &5+10\arrow[d, "f"]  \arrow[r,"\psi"] & 2+7 \arrow[d, "f"] \arrow[r, "G"] & 7\arrow[d, "f"]  \\
 3+5+5+11 \arrow[r, "\psi'"] &1+3+3+9 \arrow[r, "G'"] & 3+3+9 \arrow[r, "\psi'"] & 1+1+7 \arrow[r, "G'"] & 7 
\end{tikzcd}
\]
Notice that in the above computation we benefited from the convenient fact that $f(\pi)=\pi$ when $\pi\in \mathcal{D}\cap \mathcal{C}$, where, $\mathcal{D}$ is the set of all Schur partitions and $\mathcal{C}$ is the set of all Alladi partitions. 

\section{The Alladi-Schur polynomials}

This last section gives a treatment of the Alladi-Schur polynomials.\smallskip

 Let $\mathcal{D}_{N}(m,n)$ be the subset of $\mathcal{D}(m,n)$ with the added restriction that all the parts are less than or equal to $N$. Andrews introduced, see \cite{Andrews2019} and \cite{10.1007/978-3-319-68376-8_3}, the Alladi-Schur polynomials 
$$d_N(x)=\sum_{n,m\geq 0} D_{N}(m,n)x^mq^n,$$
where $|\mathcal{D}_{N}(m,n)|=D_{N}(m,n)$. In fact, the way he proved equation $(2)$ was by establishing some recursive relations for the Alladi-Schur polynomials \cite{Andrews2019}. In \cite{10.1007/978-3-319-68376-8_3}, he extended these recursive relations to what follows.
\begin{equation}
d_{2N}(x)=\lambda(x)d_{2N-3}(xq^2),
\end{equation}
\begin{equation}
    d_{2N-1}(x)=\lambda(x)(d_{2N-4}(xq^2)
+xq^{2N-1}(1-xq)d_{2N-7}(xq^2)),
\end{equation}
    where $N\geq3$, $d_{-1}(x)=1$, and following Andrews' notation, we denote $\lambda(x)=(1+xq+x^2q^2)$.\\

The first relation easily follows from what we have so far. In fact, it is derived bijectively in what follows.
\begin{thm}
 There is a bijection $$\phi_{2N}:\mathcal{D}_{2N}(m,n)\to \mathcal{D}_{2N-3}(m,n-2m)\cup\mathcal{D}_{2N-3}(m-1,n-2m+1)\cup\mathcal{D}_{2N-3}(m-2,n-2m+2).$$
In particular,
$$D_{2N}(m,n)=D_{2N-3}(m,n-2m)+D_{2N-3}(m-1,n-2m+1)+D_{2N-3}(m-2,n-2m+2),$$ and 
$$d_{2N}(x)=\lambda(x)d_{2N-3}(xq^2).$$

\end{thm}

\begin{proof}
Notice that the image of any even part under $\psi$ is decreased by either $4$ or $3$. So $\phi_{2N}=\phi$ works.  
\end{proof}

\bigskip

We can now handle the case of odd subscript Alladi-Schur polynomials. To do this, we need one more relation.

Let $\chi_n$ be the indicator function for the set of multiples of $n$. Then, we trivially have 
\begin{equation}
 d_N(x)=d_{N-1}(x)+x^{1+\chi_2(N)}q^{N}d_{N-3-\chi_3(N)}(x).   
\end{equation}

The bijective proof of this can be obtained by simply removing any occurrence of the part $N$.\\

 In the proofs of the cases below, unless otherwise stated, each step will utilize a result that is equivalent to at most one of equations $(4)$ or $(6)$. Since we have bijective proofs for both, the bijective proof of any of the cases below can be simply ``read off" the $q$-series proof. This will be explicitly illustrated for the case where the subscript is $=6N-1$. For all the cases below $N\geq1$, in other words, the subscript of the LHS is $\geq5$.

\bigskip

\begin{thm}

 For $N\geq3$, we have the following.

   $$ d_{2N-1}(x)=\lambda(x)(d_{2N-4}(xq^2)
+xq^{2N-1}(1-xq)d_{2N-7}(xq^2)).$$

\end{thm}
\begin{proof}
\begin{case}[subscript is $\equiv -1 \pmod{6}$]

\begin{align*}
d_{6N-1}(x)&=d_{6N}(x)-x^2q^{6N}d_{6N-4}(x)\\
&=\lambda(x)(d_{6N-3}(xq^2)
-x^2q^{6N}d_{6N-7}(xq^2))\\
&=\lambda(x)((d_{6N-4}(xq^2)+xq^{6N-1}d_{6N-7}(xq^2))
-x^2q^{6N}d_{6N-7}(xq^2))\\
&=\lambda(x)(d_{6N-4}(xq^2)
+xq^{6N-1}(1-xq)d_{6N-7}(xq^2))
\\
\end{align*}
\end{case}
\begin{case}[subscript is $\equiv 1 \pmod{6}$]

\begin{align*}
d_{6N+1}(x)&=d_{6N}(x)+xq^{6N+1}d_{6N-2}(x)\\
&=\lambda(x)(d_{6N-3}(xq^2)
+xq^{6N+1}d_{6N-5}(xq^2))\\
&=\lambda(x)((d_{6N-2}(xq^2)-x^2q^{6N+2}d_{6N-5}(xq^2))
+xq^{6N+1}d_{6N-5}(xq^2))\\
&=\lambda(x)(d_{6N-2}(xq^2)
+xq^{6N+1}(1-xq)d_{6N-5}(xq^2))
\\
\end{align*}
\end{case}
\newpage

\begin{case}[subscript is $\equiv 3 \pmod{6}$]

The proof of this case utilizes the result of case 1.

\begin{align*}
d_{6N+3}(x)
&=  d_{6N+2}(x)+xq^{6N+3}(d_{6N-1}(x))\\
&= \lambda(x)(d_{6N-1}(xq^2)
+xq^{6N+3}(d_{6N-4}(xq^2)
+xq^{6N-1}(1-xq)d_{6N-7}(xq^2)))\\
&=\lambda(x)(d_{6N-1}(xq^2)+x^2q^{6N+4}d_{6N-4}(xq^2)-x^2q^{6N+4}d_{6N-4}(xq^2) \\
&\quad +xq^{6N+3}(d_{6N-4}(xq^2)
+xq^{6N-1}(1-xq)d_{6N-7}(xq^2))) 
\\
&=\lambda(x)(d_{6N}(xq^2)
+xq^{6N+3}(1-xq)(d_{6N-4}(xq^2)
+xq^{6N-1}d_{6N-7}(xq^2)))\\
&=\lambda(x)(d_{6N}(xq^2)
+xq^{6N+3}(1-xq)d_{6N-3}(xq^2))
\end{align*}
\end{case}
\end{proof}

The attentive reader will notice that the first two cases could be merged into a single case. In addition, the first step of all these cases could be combined into a single first step. However, doing this will make the proofs much harder to read.\\

Allow me to illustrate the bijective account for the case where the subscript is $\equiv -1 \pmod{6}$. Since we need to prove an identity involving a minus sign bijectively, some rearrangement is convenient. To that end, rewrite the identity as follows. 
$$d_{6N-1}(x)=d_{6N}(x)-x^2q^{6N}d_{6N-4}(x)=\lambda(x)(d_{6N-4}(xq^2)
+xq^{6N-1}d_{6N-7}(xq^2))-x^2q^{6N}d_{6N-4}(x)$$

Now, begin by constructing a bijection $$\phi^*:\mathcal{D}_{6N}(m,n)\to I(6N,m,n),$$ where $$I(6N,m,n)=\bigcup_{i=0}^2(\mathcal{D}_{6N-4}(m-i,n-2m+i)\cup\mathcal{D}_{6N-7}(m-i-1,n-2m+i-(6N-3))),$$ defined as follows. For $\pi\in\mathcal{D}_{6N}(m,n)$ if $6N-3\in\phi(\pi)$ then $\phi^*(\pi)=\phi(\pi)\setminus\{6N-3\}$. Otherwise, $\phi^*(\pi)=\phi(\pi)$.\\ 

We immediately obtain the following bijection from the restriction. 
$$\phi^*\vert_{\mathcal{D}_{6N-1}(m,n)}:\mathcal{D}_{6N-1}(m,n)\to I(6N,m,n)\setminus\phi^*(\mathcal{D}_{6N}(m,n)\setminus\mathcal{D}_{6N-1}(m,n))$$

On the other hand, we have the following bijections.

$$\phi^*(\mathcal{D}_{6N}(m,n)\setminus\mathcal{D}_{6N-1}(m,n))\xrightarrow[]{{\phi^*}^{-1}}\mathcal{D}_{6N}(m,n)\setminus\mathcal{D}_{6N-1}(m,n)\xrightarrow[]{\pi\to\pi \setminus\{6N\}}\mathcal{D}_{6N-4}(m-2,n-6N) $$\\
In other words, the restriction $\phi^*\vert_{\mathcal{D}_{6N-1}(m,n)}$ along with the above bijections for\\$\phi^*(\mathcal{D}_{6N}(m,n)\setminus\mathcal{D}_{6N-1}(m,n))$ give a complete bijective account of the identity 

$$d_{6N-1}(x)=\lambda(x)(d_{6N-4}(xq^2)
+xq^{6N-1}d_{6N-7}(xq^2))-x^2q^{6N}d_{6N-4}(x).$$

If one prefers the form
$$d_{6N-1}(x)=\lambda(x)(d_{6N-4}(xq^2)
+xq^{6N-1}(1-xq)d_{6N-7}(xq^2)),$$
then one more application of $\phi$ on $\mathcal{D}_{6N-4}(m-2,n-6N)$ gives just that. \\

The bijective accounts of the $1,3\equiv \pmod{6}$ are achieved analogously.

\begin{remark}
We note that in \cite{ANDREWS2022}, Andrews, Chern, and Li gave an analytic counterpart of Andrews' refinement of the Alladi-Schur theorem. They gave both a $q$-hypergeometric proof and a computer-assisted proof. However, the techniques in \cite{ANDREWS2022} are categorically different from the work presented in this paper and have not influenced the proofs in any way. This is clear for the bijective portion of this paper, but it is even true for the parts involving the Alladi-Schur polynomials. Indeed, despite the fact that Andrews originally established equation (1) using the Alladi-Schur polynomials, the paper \cite{ANDREWS2022} does not mention the Alladi-Schur polynomials.
\end{remark}

\section*{Acknowledgment}
I would like to thank my doctoral advisor, Dr. Krishnaswami Alladi, for carefully reviewing my drafts and making suggestions that made my work clearer and more readable. I would also like to thank Saurabh Anand for his thoughtful suggestions regarding computer verification. In addition, I extend my thanks to Ae Ja Yee for her suggestion regarding the presentation of this paper and to George Andrews for his encouraging remarks. Lastly, I greatly appreciate the referee's helpful suggestions.  

\newpage
\appendix
\section{A couple of alternative perspectives}
This appendix provides an alternative definition of an upper minimal segment and an alternative representation of my bijection $f$.

\begin{definition}
An \underline{upper minimal segment} $\Sigma$ is a sub-partition of a Schur partition $\pi$ that, in the generic upper factorization of $\pi$, begins with an odd singleton $>3$ and ends with an even singleton, where every even singleton of $\Sigma$ is larger than its largest odd singleton. In addition, the following properties hold.\\

1-The gap between consecutive singletons of the same parity is $4$.\\

2-If an odd singleton immediately precedes an even singleton, the gap is $5$.\\

3-If an odd singleton immediately precedes a pair or an even singleton immediately succeeds a pair, then the gap is $4$.\\

4-If an even singleton immediately precedes a pair or an odd singleton immediately succeeds a pair, then the gap is $3$.  \\

5-The gap between consecutive pairs is $3$.
\end{definition}
The reader will see that this definition is equivalent to the in-text definition. Indeed, primitive upper minimal segments abide by the above list, and adding a prefix or suffix does not violate the rules. Likewise, it can be shown inductively that the above partitions are contained in the set of upper minimal segments.\\

In private communications \cite{KA}, K. Alladi has pointed out to me that the operation $\phi'$ has a nice representation in $2$-modular graphs. Namely, it corresponds to removing the leftmost column. Here, the rows correspond to the parts, and the top row is the largest part. To this end, we can represent the steps of $f$ as sequences of pairs $(\pi_i,\pi'_i)\in \mathcal{D}\times \mathcal{C}$ where $\pi_{i+1}=\phi(\pi_i)$ and $\pi'_{i+1}$ is obtained by adding a rightmost column of length $m(\pi_i)$, with the bottom $g(\pi_i)$ entries being $1$ and the remaining entries being $2$, to the  $2$-modular graph $\pi'_i$. An example is illustrated below.
\[
7+[11+14]+18\quad\begin{array}{|c}
\text{empty}
\end{array}\quad \xrightarrow[g(\pi)=0]{\phi}\quad
4+[8+11]+15\quad\begin{array}{|c}
2 \\
2 \\
 2 \\
 2 \\
 2\\
2\\
\end{array}\]

\[\xrightarrow[g(\pi)=1]{\phi}\quad
1+[5+8]+13\quad\begin{array}{|cc}
2 & 2 \\
2 & 2\\
 2 & 2\\
 2 & 2\\
 2 & 2\\
2 & 1\\
\end{array}\quad \xrightarrow[g(\pi)=1]{\phi}\quad
[2+5]+11\quad\begin{array}{|ccc}
2 & 2 &2 \\
2 & 2 &2\\
2 & 2 &2 \\
2 & 2 &2 \\
2 & 2 &1\\
2 & 1\\
\end{array}\]

\[\xrightarrow[g(\pi)=2]{\phi}\quad
3+9\quad\begin{array}{|cccc}
2 & 2 &2&2 \\
2 & 2 &2&2\\
2 & 2 &2&1 \\
2 & 2 &2&1 \\
2 & 2 &1\\
2 & 1\\
\end{array}\quad \dots \quad
\text{empty}\quad\begin{array}{|ccccccccc}
2 & 2 &2&2&2&2&2&2&1 \\
2 & 2 &2&2&2&1\\
2 & 2 &2&1 \\
2 & 2 &2&1 \\
2 & 2 &1\\
2 & 1\\
\end{array}\]

Here, the dots correspond to the fact that $\phi$ and $\phi'$ agree in $ \mathcal{D}\cap \mathcal{C}$. This means we can skip the remaining steps at that point since we know that the final step will just be $$\pi_i|\pi'_i\dots \text{empty}|\pi_i+\pi'_i.$$
The reader is invited to perform the remaining steps and verify this. The notation $\pi_i+\pi'_i$ means that the top row (the largest part) of $\pi_i$ is added to the top row of $\pi_i'$ and the second row to the second row, and so on.

\bibliographystyle{amsplain}

\end{document}